\RequirePackage{fix-cm}
\documentclass[smallextended]{svjour3-arxiv}       
\smartqed  
\usepackage{graphicx}
\usepackage{hyperref}
\usepackage{amssymb}
\usepackage{amsmath}

\usepackage{mathptmx}      
\renewcommand{\qed}{\hfill$\boxtimes\hspace{-1.725ex}\boxplus$}

\newcommand{\qedef}{\hfill$\triangle\hspace{-0.695em}\diamond$}

\begin{document}

\title{\itshape{On Decidability of  the Ordered Structures of Numbers}\thanks{This is  a part of the Ph.D. thesis of the first author written under the supervision of the second author who is partially supported by grant  ${\sf N}^{\sf o}$ 90030053 of the Institute for Research in Fundamental Sciences ($\mathbb{IPM}$), Tehran, {\sc Iran}.}
}


\author{{\bfseries {\scshape Ziba Assadi}}         \  {\bfseries {\scshape \&}} \
        {\bfseries {\scshape Saeed Salehi}} 
}

\authorrunning{\sc Ziba Assadi \ \& \  Saeed Salehi} 

\institute{Z. Assadi \at
Department of Mathematics, University of Tabriz, 29 Bahman Blvd., P.O.Box~51666--16471, Tabriz, {\sc Iran}. \\  \email{z\_assadi.golzar@tabrizu.ac.ir}           
           \and
          S. Salehi \at
              Research Institute for Fundamental Sciences (RIFS), University of Tabriz, 29 Bahman Blvd., P.O.Box~51666--16471, Tabriz, {\sc Iran}.   \qquad   \email{salehipour@tabrizu.ac.ir} \\  School of Mathematics,  Institute for Research in Fundamental Sciences ($\mathbb{IPM}$), P.O.Box 19395--5746,  Niavaran, Tehran, {\sc Iran}. \\    \email{saeedsalehi@ipm.ir}  \hfill
              Web: \url{http://saeedsalehi.ir}
}

\date{}

\maketitle

\begin{abstract}
The  ordered structures  of natural, integer, rational and real numbers are studied here. It is  known that the theories of these numbers in the language of order are decidable and finitely axiomatizable. Also, their theories in the language of order and addition are decidable and infinitely axiomatizable. For the language of order and multiplication, it is known that the theories of $\mathbb{N}$ and $\mathbb{Z}$ are not decidable (and so not axiomatizable by any  computably enumerable set of sentences). By Tarski's theorem, the multiplicative ordered structure of $\mathbb{R}$ is decidable also; here we prove this result directly and present an axiomatization. The structure of $\mathbb{Q}$ in the language of order and multiplication seems to be missing in the literature; here we show the decidability of its theory by the technique of quantifier elimination and after presenting an infinite axiomatization for this structure we prove that it is not finitely axiomatizable.
\end{abstract}

%

\section{Introduction and Preliminaries}\label{sec-intro}
{\em Entscheidungsproblem}, one of the fundamental problems of (mathematical) logic, asks for a
single-input Boolean-output algorithm that takes a formula $\varphi$ as input and outputs {\tt `yes'} if $\varphi$ is logically valid and outputs {\tt `no'} otherwise.
 Now, we know that this problem is not (computably) solvable.  One reason for this is the existence of an essentially undecidable and finitely axiomatizable theory, see e.g.~\cite{visser}; for another proof see~\cite[Theorem 11.2]{bbj}. However, by G\"odel's completeness theorem, the set of logically valid formulas is computably enumerable, i.e., there exists an input-free algorithms that (after running) lists all the valid formulas (and nothing else).
For the structures, since their theories are complete, the story is different: the theory of a structure is either decidable or that structure is not axiomatizable (by any computably enumerable set of sentences; see e.g.~\cite[Corollaries 25G and 26I]{enderton} or~\cite[Theorem 15.2]{monk}).
For example, the additive theory of natural numbers $\langle\mathbb{N};+\rangle$
    was shown to be decidable by Presburger in 1929
     (and by Skolem in 1930; see \cite{smorynski}).
     The multiplicative theory of the natural numbers
     $\langle\mathbb{N};\times\rangle$ was announced  to be decidable by Skolem in 1930. Then it was expected that the theory of addition and multiplication of natural numbers
would be decidable too; confirming Hilbert's Program. But the world was shocked in 1931
 by G\"odel's incompleteness theorem which implies that  the theory of
  $\langle\mathbb{N};+,\times\rangle$ is undecidable (see the subsection~\ref{subsec-n} below).
In this paper we study the theories of the sets $\mathbb{N}$, $\mathbb{Z}$, $\mathbb{Q}$ and $\mathbb{R}$ in the languages $\{<\}$, $\{<,+\}$ and $\{<,\times\}$; see the table below.

\begin{center}
\begin{tabular}{|c||c|c|c|c|}
\hline
  & $\mathbb{N}$ & $\mathbb{Z}$ & $\mathbb{Q}$ & $\mathbb{R}$ \\
\hline
\hline
$\{<\}$ & Thm.~\ref{thm-on} & Thm.~\ref{thm-oz} & Thm.~\ref{thm-o} & Thm.~\ref{thm-o} \\
\hline
$\{<,+\}$ & Rem.~\ref{rem-noa} & Thm.~\ref{thm-oaz} & Thm.~\ref{thm-oa} & Thm.~\ref{thm-oa} \\
\hline
$\{<,\times\}$ & $\S$~\ref{subsec-n} & $\S$~\ref{subsec-z} & Thm.~\ref{omq+} & Thm.~\ref{thm-or} \\
\hline
\hline
$\{+,\times\}$ & $\S$~\ref{subsec-n} & $\S$~\ref{subsec-z} &   $\S$~\ref{subsec-q} & $\S$~\ref{subsec-r} \\
\hline
\end{tabular}
\end{center}

Let us note that order is definable in the language $\{+,\times\}$ in these sets: in $\mathbb{N}$ by
  $x<y\iff \exists z (z\!+\!z \neq z\wedge x\!+\!z\!=\!y)$, and in $\mathbb{Z}$ by Lagrange's four square theorem $x<y$ is equivalent with    $\exists t,u,v,w (x\!\neq\!y \wedge x\!+\!t\!\cdot\!t\!+\!u\!\cdot\!u\!+\!v\!\cdot\!v\!+\!w\!\cdot\!w =y).$ The four square  theorem holds in $\mathbb{Q}$ too: for any   $p/q\!\in\!\mathbb{Q}^+$ we have $pq\!>\!0$ so  $pq\!=\!a^2\!+\!b^2\!+\!c^2\!+\!d^2$ for some integers $a,b,c,d$; therefore, 
  $p/q\!=\!pq/q^2\!=\!(a/q)^2\!+\!(b/q)^2\!+\!(c/q)^2\!+\!(d/q)^2$ holds. Thus, the same formula defines the order ($x<y$) in $\mathbb{Q}$ as well. Finally, in   $\mathbb{R}$  the relation      $x<y$ is equivalent with the formula $\exists z (z\!+\!z\!\neq\!z\wedge x+z\!\cdot\!z=y)$.

The decidability of $\mathbb{N},\mathbb{Z},\mathbb{Q},\mathbb{R}$ in the languages $\{<\}$ and $\{<,+\}$ is already known.  It is also known that the theories of   $\mathbb{N}$ and $\mathbb{Z}$ in the language $\{<,\times\}$ are undecidable. The theory of    $\mathbb{R}$ in the language $\{<,\times\}$ is decidable  too by Tarski's theorem (which states the decidability of the theory of  $\langle\mathbb{R};<,+,\times\rangle$). Here, we prove this directly by presenting an explicit axiomatization. Finally, the structure $\langle\mathbb{Q};<,\times\rangle$ is studied in this paper (seemingly, for the first time). We show, by the method of quantifier elimination, that the theory of this structure is decidable. Here,   the (super-)structure  $\langle\mathbb{Q};+,\times\rangle$ is not usable since it is undecidable (proved by Robinson~\cite{robinson}; see also \cite[Theorem 8.30]{smorynski}). On the other hand its (sub-)structure $\langle\mathbb{Q};\times\rangle$ is decidable (proved in~\cite{mostowski} by Mostowski; see also~\cite{salehi-m}). So, the three structures  $\langle\mathbb{Q};+,\times\rangle$ and $\langle\mathbb{Q};<,\times\rangle$ and $\langle\mathbb{Q};\times\rangle$ are different from each other; the order relation  $<$ is not definable in $\langle\mathbb{Q};\times\rangle$ and the addition operation $+$ is not definable in $\langle\mathbb{Q};<,\times\rangle$ (by our results). This paper is a continuation of the conference paper~\cite{salehi-mo}.

\section{The Ordered Structure of Numbers}\label{sec-order}

\begin{definition}[Ordered Structure]\label{def-os}
An {\em ordered structure} is a triple $\langle A;<,\mathcal{L}\rangle$ where $A$ is a non-empty set and $<$ is a binary relation on $A$ which satisfies the following axioms:

\begin{tabular}{l}
($\texttt{O}_1$) \; $\forall x,y(x<y\rightarrow y\not<x)$,  \\ ($\texttt{O}_2$) \; $\forall x,y,z (x<y<z\rightarrow x<z)$ and \\
($\texttt{O}_3$) \; $\forall x,y (x<y \vee x=y \vee y<x)$;
\end{tabular}

\noindent
and $\mathcal{L}$ is a language.  \qedef
\end{definition}
Here $\mathcal{L}$ could be empty, or any language, for example $\{+\}$ or $\{\times\}$ or $\{+,\times\}$.
\begin{definition}[Various Types of Orders]\label{def-vto}
A linear  order relation $<$ is called {\em dense} if it satisfies

\begin{tabular}{l}
($\texttt{O}_4$) \; $\forall x,y (x<y\rightarrow\exists z [x<z<y])$.
\end{tabular}

\noindent An order relation $<$ is called
{\em without endpoints} if it satisfies

\begin{tabular}{l}
 ($\texttt{O}_5$) \; $\forall x\exists y (x<y)$  and  \\ ($\texttt{O}_6$) \; $\forall x\exists y (y<x)$.
\end{tabular}

\noindent A {\em discrete} order has the property that   any element has an immediate successor (i.e., there is no other element in between them). If the successor of $x$ is denoted by $\mathfrak{s}(x)$ then a discrete order satisfies

\begin{tabular}{l}
($\texttt{O}_7$) \; $\forall x,y (x\!<\!y \;\leftrightarrow\; \mathfrak{s}(x)\!<\!y \vee \mathfrak{s}(x)\!=\!y)$.
\end{tabular}

\noindent The successor of an integer $x$ is $\mathfrak{s}(x)=x+1$.
\qedef\end{definition}

\begin{remark}[\textbf{The Main Lemma of Quantifier Elimination}]\label{mainlem}
It is known that a theory (or a structure) admits quantifier elimination if and only if every formula of the form $\exists x (\bigwedge\hspace{-1.5ex}\bigwedge_{i}\alpha_i)$ is equivalent with  a quantifier-free formula, where each $\alpha_i$ is either an atomic formula or the negation of an atomic formula.
This has been proved in e.g.
\cite[Theorem 31F]{enderton},
\cite[Lemma 2.4.30]{hinman},
\cite[Theorem 1, Chapter 4]{kk},
\cite[Lemma 3.1.5]{marker}
and \cite[Lemma 4.1]{smorynski}.
In the presence of a linear order relation ($<$) by the equivalences $(s\neq t) \leftrightarrow (s<t \vee t<s)$ and  $(s\not < t) \leftrightarrow (t\leqslant s)$, which follow from the axioms  $\{\texttt{O}_1,\texttt{O}_2,\texttt{O}_3\}$ (of Definition~\ref{def-os}), we do not need to consider the negated atomic formulas (when there is no   relation symbol in the language other than $<,=$).
\qedef\end{remark}

\paragraph[]{\bf Convention:}
Let $\bot$ denote the (propositional constant of) contradiction, and $\top$ the truth. By convention, $a\leqslant b$ abbreviates  $a<b \vee a=b$.
The symbols $\times$ and $\cdot$ are used interchangeably throughout the paper. For convenience, let us agree that  $0^{-1}=0$ as this does not contradict our intuition. Needless to say, $x^n$ symbolizes $x\cdot x \cdot \ldots \cdot x$ ($n-$times); also $x+x+\cdots +x$ ($n-$times) is abbreviated as $n\centerdot x$.\qedef

\bigskip

The following theorem has been proved in
\cite[Theorems 2.4.1 and 3.1.3]{marker}.
 Here, we present a syntactic (proof-theoretic) proof.

\begin{theorem}[Axiomatizablity of $\langle\mathbb{R};<\rangle$ and $\langle\mathbb{Q};<\rangle$]\label{thm-o}
The    theory axiomatized by the finite set $\{\texttt{O}_1,\texttt{O}_2,\texttt{O}_3,
\texttt{O}_4,\texttt{O}_5,\texttt{O}_6\}$ \textup{(}i.e., the theory of dense linear orders without endpoints, see  Definitions~\ref{def-os} and~\ref{def-vto}\textup{)}
completely axiomatizes the order theory of the   real and rational  numbers and, moreover, the structure $\langle\mathbb{R};<\rangle$ \textup{(}and also $\langle\mathbb{Q};<\rangle$\textup{)} admits quantifier elimination, and so its theory is decidable.
\end{theorem}
\begin{proof}
All the atomic formulas are either of the form $u<v$ or $u=v$ for some variables $u$ and $v$. If both of the variables are equal then $u<u$ is equivalent with $\bot$  by $\texttt{O}_1$ and $u=u$ is equivalent with $\top$.
So, by Remark~\ref{mainlem},  it suffices to eliminate the quantifier of the formulas of the form
\begin{equation}\label{f-1}
\exists x (\bigwedge\hspace{-2.15ex}\bigwedge_{i<\ell} y_i<x \wedge \bigwedge\hspace{-2.55ex}\bigwedge_{j<m} x<z_j \wedge \bigwedge\hspace{-2.35ex}\bigwedge_{k<n} x=u_k)
\end{equation}
where $y_i$'s, $z_j$'s and $u_k$'s are variables. Now, if $n\neq 0$ then the formula~\eqref{f-1} is equivalent with the quantifier-free formula
\begin{equation*}
\bigwedge\hspace{-2.15ex}\bigwedge_{i<\ell} y_i<u_0 \wedge \bigwedge\hspace{-2.55ex}\bigwedge_{j<m} u_0<z_j \wedge \bigwedge\hspace{-2.35ex}\bigwedge_{k<n} u_0=u_k.
\end{equation*}
So, let us suppose that $n=0$. Then if $m=0$ or  $\ell=0$  the formula~\eqref{f-1} is equivalent with the quantifier-free formula $\top$, by the axioms   $\texttt{O}_5$ and $\texttt{O}_6$  (with $\texttt{O}_2$ and $\texttt{O}_3$) respectively,  and if $\ell,m\neq 0$ it is equivalent with the quantifier-free formula $\bigwedge\hspace{-1.55ex}\bigwedge_{i<\ell,j<m} y_i<z_j$ by the axiom  $\texttt{O}_4$ (with $\texttt{O}_2$ and $\texttt{O}_3$).
\qed\end{proof}

In fact for any set $A$ such that $\mathbb{Q}\subseteq A\subseteq\mathbb{R}$ the structure $\langle A;<\rangle$ can be completely axiomatized by the finite set of axioms $\{\texttt{O}_1,\texttt{O}_2,\texttt{O}_3,\texttt{O}_4,
\texttt{O}_5,\texttt{O}_6\}$ in Definitions~\ref{def-os} and~\ref{def-vto}.

The theory of the structure $\langle\mathbb{Z};<\rangle$ does not admit quantifier elimination: for example the formula $\exists x (y<x<z)$ is not equivalent with any quantifier-free formula in the language $\{<\}$ (note that it is not equivalent with $y<z$). If we add the successor operation $\mathfrak{s}$ to the language then that formula will be equivalent with $\mathfrak{s}(y)<z$ and the process of quantifier elimination will go through.

\begin{theorem}[Axiomatizablity of $\langle\mathbb{Z};<\rangle$]\label{thm-oz}
The  finite  theory of discrete  linear orders without endpoints, consisting of the axioms    $\texttt{O}_1$, $\texttt{O}_2$, $\texttt{O}_3$, $\texttt{O}_7$ plus

\begin{tabular}{l}
 ($\texttt{O}_8$) \; $\forall x\exists y (\mathfrak{s}(y)=x)$
\end{tabular}

\noindent completely axiomatizes the order theory of the   integer   numbers and, moreover, the structure $\langle\mathbb{Z};<,\mathfrak{s}\rangle$ admits quantifier elimination, and so its theory is decidable.
\end{theorem}
\begin{proof}
We  note that all the terms in the language $\{<,\mathfrak{s}\}$ are of the form $\mathfrak{s}^n(y)$ for some variable $y$ and $n\in\mathbb{N}$. So, all the atomic formulas are either of the form    $\mathfrak{s}^n(u)=\mathfrak{s}^m(v)$ or $\mathfrak{s}^n(u)<\mathfrak{s}^m(v)$ for some variables $u,v$. If a variable $x$ appears in the both sides of an atomic formula, then we have either $\mathfrak{s}^n(x)=\mathfrak{s}^m(x)$ or $\mathfrak{s}^n(x)<\mathfrak{s}^m(x)$. The formula $\mathfrak{s}^n(x)=\mathfrak{s}^m(x)$ is equivalent with $\top$ when $n=m$ and with $\bot$ otherwise; also $\mathfrak{s}^n(x)<\mathfrak{s}^m(x)$ is equivalent with $\top$ when $n<m$ and with $\bot$ otherwise. So, it suffices to consider the atomic formulas of the form  $t<\mathfrak{s}^n(x)$ or $\mathfrak{s}^n(x)<t$ or $\mathfrak{s}^n(x)=t$ for some $x$-free term $t$ and $n\in\mathbb{N}^+$. Now, by Remark~\ref{mainlem}, we eliminate the quantifier of the formulas
\begin{equation}\label{f-2}
\exists x (\bigwedge\hspace{-2.15ex}\bigwedge_{i<\ell} t_i<\mathfrak{s}^{p_i}(x) \wedge \bigwedge\hspace{-2.55ex}\bigwedge_{j<m} \mathfrak{s}^{q_j}(x) <s_j \wedge \bigwedge\hspace{-2.35ex}\bigwedge_{k<n} \mathfrak{s}^{r_k}(x)=u_k).
\end{equation}
The axioms prove 
$[a<b] \leftrightarrow [\mathfrak{s}(a)<\mathfrak{s}(b)]$ and   $[a=b] \leftrightarrow [\mathfrak{s}(a)=\mathfrak{s}(b)]$;   so
we can assume that $p_i$'s and $q_j$'s and $r_k$'s in
the formula~\eqref{f-2} are equal to each other, say to $\alpha$. Then by $\texttt{O}_8$ the formula~\eqref{f-2} is equivalent with
\begin{equation}\label{f-3}
\exists y (\bigwedge\hspace{-2.15ex}\bigwedge_{i<\ell} t_i'<y \wedge \bigwedge\hspace{-2.55ex}\bigwedge_{j<m} y<s_j' \wedge \bigwedge\hspace{-2.35ex}\bigwedge_{k<n} y=u_k')
\end{equation}
for some (possibly new)  terms $t_i',s_j',u_k'$ (and   $y=\mathfrak{s}^\alpha(x)$). Now, if $n\neq 0$ then the formula \eqref{f-3} is equivalent with the quantifier-free formula
\begin{equation*}
\bigwedge\hspace{-2.15ex}\bigwedge_{i<\ell} t_i'<u_0' \wedge \bigwedge\hspace{-2.55ex}\bigwedge_{j<m} u_0'<s_j' \wedge \bigwedge\hspace{-2.35ex}\bigwedge_{k<n} u_0'=u_k'.
\end{equation*}
Let us then assume that $n=0$. The formula
\begin{equation}\label{f-4}
\exists x (\bigwedge\hspace{-2.15ex}\bigwedge_{i<\ell} t_i<x \wedge \bigwedge\hspace{-2.55ex}\bigwedge_{j<m} x<s_j)
\end{equation}
is equivalent with the quantifier-free formula
$\bigwedge\hspace{-1.55ex}\bigwedge_{i<\ell,j<m} \mathfrak{s}(t_i)<s_j$
by the axiom $\texttt{O}_7$ (in Definition~\ref{def-vto}).
\qed\end{proof}

The structure $\langle\mathbb{N};<\rangle$ can also be finitely axiomatized. The following theorem has been proved in \cite[Theorem~32A]{enderton}   so we do not present its proof here.

\begin{theorem}[Axiomatizablity of $\langle\mathbb{N};<\rangle$]\label{thm-on}
The   finite theory consisting of the axioms $\{\texttt{O}_1,\texttt{O}_2,\texttt{O}_3,\texttt{O}_7\}$ \textup{(}in Definitions~\ref{def-os} and~\ref{def-vto}\textup{)} and also the following two axioms

\begin{tabular}{l}
 ($\texttt{O}_8^\circ$) \; $\forall x\exists y (x\neq {\bf 0} \rightarrow \mathfrak{s}(y)=x)$, \\
 ($\texttt{O}_9$) \; $\forall x (x\not<{\bf 0})$,
\end{tabular}

\noindent completely axiomatizes the order theory of the   natural   numbers and, moreover, the structure $\langle\mathbb{N};<,\mathfrak{s},{\bf 0}\rangle$ admits quantifier elimination, and so its theory is decidable.
\qed
\end{theorem}
Let us note that the structure $\langle\mathbb{N};<,\mathfrak{s} \rangle$ does not admit quantifier elimination, since e.g. the formula $\exists x (\mathfrak{s}(x)=y)$ is not equivalent with any quantifier-free formula in the language $\{<,\mathfrak{s}\}$. However, this formula is equivalent with ${\bf 0}<y$.

\section{The  Additive Ordered Structures of Numbers}\label{sec-oa}
Here we study the structures of the sets $\mathbb{N},\mathbb{Z},\mathbb{Q},\mathbb{R}$ over the language $\{+,<\}$.
\begin{definition}[Some Group Theory]\label{def-g}
A {\em group} is a structure  $\langle G;\ast,{\sf e},\iota\rangle$ where $\ast$ is a binary operation on $G$, ${\sf e}$ is a constant (a special element of $G$) and $\iota$ is a unary operation on $G$ which satisfy the following axioms:

\begin{tabular}{l}
 $\forall x,y,z\,[x\ast (y\ast z)=(x\ast y)\ast z]$; \\
 $\forall x (x\ast {\sf e}=x)$; \\
 $\forall x (x\ast  \iota(x)={\sf e})$.
\end{tabular}

\noindent
It is called an {\em abelian} group when it also satisfies

\begin{tabular}{l}
 $\forall x,y (x\ast y=y\ast x)$.
\end{tabular}

\noindent
A group is called {\em non-trivial} when

\begin{tabular}{l}
 $\exists x (x\neq {\sf e})$;
\end{tabular}

\noindent and it is called {\em divisible} when for $n\in\mathbb{N}$ we have

\begin{tabular}{l}
 $\forall x\exists y[x=\ast^n(y)]$
\end{tabular}

\noindent
where $\ast^n(y)=y\ast\cdots\ast y \; (n-\textrm{times})$.

\noindent An {\em ordered group} is a group equipped with an order relation $<$ (which satisfies $\texttt{O}_1,\texttt{O}_2,\texttt{O}_3$) such that also the axiom

\begin{tabular}{l}
$\forall x,y,z(x\!<\!y \;\rightarrow\; x\ast z\!<\!y\ast z \;\wedge\; z\ast x\!<\!z\ast y)$
\end{tabular}

\noindent
is satisfied in it.
\qedef\end{definition}

\noindent
The following has been proved in e.g.
\cite[Corollary 3.1.17]{marker}:

\begin{theorem}[Axiomatizablity of $\langle\mathbb{R};<,+\rangle$ and $\langle\mathbb{Q};<,+\rangle$]\label{thm-oa}
The following infinite theory \textup{(}of non-trivial ordered divisible abelian groups\textup{)}
completely axiomatizes the order and additive theory of the   real and rational  numbers and, moreover, the structure $\langle\mathbb{R};<,+,-,{\bf 0}\rangle$ \textup{(}and also $\langle\mathbb{Q};<,+,-,{\bf 0}\rangle$\textup{)} admits quantifier elimination, and so its theory is decidable.

\begin{tabular}{l}
($\texttt{O}_1$) \; $\forall x,y(x<y\rightarrow y\not<x)$  \\ ($\texttt{O}_2$) \; $\forall x,y,z (x<y<z\rightarrow x<z)$ \\
($\texttt{O}_3$) \; $\forall x,y (x<y \vee x=y \vee y<x)$ \\
($\texttt{A}_1$) \; $\forall x,y,z\,(x+(y+z)=(x+y)+z)$ \\
($\texttt{A}_2$) \; $\forall x (x+\mathbf{0}=x)$ \\
($\texttt{A}_3$) \; $\forall x (x+ (-x)=\mathbf{0})$ \\
($\texttt{A}_4$) \; $\forall x,y (x+y=y+x)$ \\
($\texttt{A}_5$) \; $\forall x,y,z(x<y\rightarrow x+z<y+z)$ \\
($\texttt{A}_6$) \; $\exists y (y\neq {\bf 0})$ \\
($\texttt{A}_7$) \; $\forall x\exists y(x=n\centerdot y)$  \qquad \qquad \qquad  $n\in\mathbb{N}^+$  \\
\end{tabular}

\end{theorem}
\begin{proof}
Firstly, let us note that  $\texttt{O}_4$, $\texttt{O}_5$ and $\texttt{O}_6$ can be proved from the presented axioms: if $a<b$ then by $\texttt{A}_7$ there exists some $c$ such that $c+c=a+b$;  one can easily show that $a<c<b$ holds.  Thus $\texttt{O}_4$ is proved; for  $\texttt{O}_5$ note that for any ${\bf 0}<a$ we have $a<a+a$ by $\texttt{A}_5$. A dual argument can prove the axiom $\texttt{O}_6$. Also,  the     equivalences
\begin{itemize}\itemindent=5ex
\item[(i)] $[a<b] \leftrightarrow [n\centerdot a< n\centerdot b]$ and
\item[(ii)] $[a=b] \leftrightarrow [n\centerdot a=n\centerdot b]$
\end{itemize}
can be proved from the axioms: (i) follows from $\texttt{A}_5$ (with $\texttt{O}_1,\texttt{O}_2,\texttt{O}_3$) and (ii) follows from $\forall x (n\centerdot x={\bf 0}\rightarrow x={\bf 0})$ which is derived from $\texttt{A}_5$ (with $\texttt{O}_1,\texttt{O}_2,\texttt{O}_3$).

Secondly, every term containing $x$ is equal to $n\centerdot x + t$ for some $x$-free term $t$ and $n\!\in\!\mathbb{Z}\!-\!\{0\}$. So, every atomic formula containing $x$ is equivalent with $n\centerdot x \Box t$ where $\Box\!\in\!\{=,<,>\}$. Whence, by Remark~\ref{mainlem}, it suffices to prove the equivalence of the formula
\begin{equation}\label{f-5}
\exists x (\bigwedge\hspace{-2.15ex}\bigwedge_{i<\ell} t_i<p_i\centerdot x \wedge \bigwedge\hspace{-2.55ex}\bigwedge_{j<m} q_j\centerdot x <s_j \wedge \bigwedge\hspace{-2.35ex}\bigwedge_{k<n} r_k\centerdot x=u_k)
\end{equation}
with a quantifier-free formula.
  By the equivalences (i) and (ii) above
we can assume that $p_i$'s and $q_j$'s and $r_k$'s in
the formula~\eqref{f-5} are equal to each other, say to $\alpha$. Then by $\texttt{A}_7$ the formula~\eqref{f-5} is equivalent with
\begin{equation}\label{f-6}
\exists y (\bigwedge\hspace{-2.15ex}\bigwedge_{i<\ell} t_i'<y \wedge \bigwedge\hspace{-2.55ex}\bigwedge_{j<m} y<s_j' \wedge \bigwedge\hspace{-2.35ex}\bigwedge_{k<n} y=u_k')
\end{equation}
for some (possibly new) terms $t_i',s_j',u_k'$ (and   $y=\alpha\centerdot x$). Now, the quantifier of this formula can be eliminated just like the way that the quantifier of the formula~\eqref{f-1} was eliminated in the proof of Theorem~\ref{thm-o}.
\qed\end{proof}

\begin{remark}[\textbf{Infinite Axiomatizablity}]\label{rem-infq}
To see that 
$\langle\mathbb{R};<,+\rangle$ and $\langle\mathbb{Q};<,+\rangle$ are not finitely axiomatizable, it suffices to note that for a given natural number $N$, the set $\mathbb{Q}/N!=\{\dfrac{m}{(N!)^k}\mid m\in\mathbb{Z},k\in\mathbb{N}\}$ of rational numbers, where $N!=1\times 2\times 3\times \cdots \times N$,  is closed under addition  and so satisfies the axioms $\texttt{O}_1$, $\texttt{O}_2$, $\texttt{O}_3$, $\texttt{A}_1$, $\texttt{A}_2$, $\texttt{A}_3$, $\texttt{A}_4$, $\texttt{A}_5$,  $\texttt{A}_6$ and the finite number of the instances of the axiom $\texttt{A}_7$ (for $n=1,\cdots,N$) but does not satisfy the instance of $\texttt{A}_7$ for $n={p}$ where ${p}$ is a prime number larger than $N!$.
\qedef\end{remark}

For  eliminating the quantifiers of the formulas of the structure  $\langle\mathbb{Z};<,+\rangle$ we add the (binary) congruence relations $\{\equiv_{n}\}_{n\geqslant 2}$ (modulo  standard natural numbers) to the language; let us note that $a\equiv_n b$ is equivalent with $\exists x (a+n\centerdot x = b)$.
About these congruence relations the following Generalized Chinese Remainder Theorem will be useful later; below we present a proof of this theorem from \cite{frankel}.

\begin{proposition}[Generalized Chinese Remainder]\label{crt}
For  integers $n_0,n_1,\cdots,n_k\geqslant 2$ and $t_0,t_1,\cdots,t_k$ there exists some $x$ such that  $x\equiv_{n_i} t_i$ for $i=0,\cdots,k$ if and only if  $t_i \equiv_{d_{i,j}} t_j$ holds for each $0\leqslant i<j\leqslant k$, where $d_{i,j}$ is the greatest common divisor of $n_i$ and $n_j$.
\end{proposition}
\begin{proof} The `only if' part is easy. We prove the `if' part by induction on $k$. For $k=0$ there is nothing to prove, and for $k=1$ we note that by B\'{e}zout's Identity there are $a_0,a_1$ such that $a_0n_0 + a_1n_1 = d_{0,1}$. Also, by the assumption there exists some $c$ such that $t_0-t_1=cd_{0,1}$. Now, if we take $x$ to be $a_0(n_0/d_{0,1})t_1 + a_1(n_1/d_{0,1})t_0$ then we have $x=t_0-a_0n_0c$ and  $x=t_1+a_1n_1c$ so  $x\equiv_{n_0}t_0$ and $x\equiv_{n_1}t_1$ hold. For the induction step ($k+1$) suppose that $x\equiv_{n_i} t_i$ holds for $i=0,\cdots,k$ (and that $t_i \equiv_{d_{i,j}} t_j$ holds for each $0\leqslant i<j\leqslant k+1$). Let $n$ be the least common multiplier of $n_0,\cdots,n_k$; then the greatest common divisor  $m$  of $n$ and $n_{k+1}$ is the least common multiplier of $d_{0,k+1},\cdots,d_{k,k+1}$. Now $x\equiv_{d_{i,k+1}}t_i$ holds for $0\leqslant i\leqslant k$ and so by the assumption $t_i\equiv_{d_{i,k+1}}t_{k+1}$ we have $x\equiv_{d_{i,k+1}}t_{k+1}$ (for $i=0,\cdots,k$). Therefore, $x\equiv_m t_{k+1}$ and so $x-t_{k+1}=mc$ for some $c$. By B\'{e}zout's Identity there are $a,b$ such that $an+bn_{k+1}=m$. Now, for $y=x-anc$ we have $y=t_{k+1}+bn_{k+1}c\equiv_{n_{k+1}}t_{k+1}$ and also $y\equiv_{n_i}x\equiv_{n_i}t_i$ holds for each $0\leqslant i\leqslant k$. This proves the desired conclusion.
\qed\end{proof}

The following theorem
 has been proved, in various formats, in e.g.  \cite[Chapter 24]{bbj}, \cite[Theorem 32E]{enderton}, \cite[Corollary 2.5.18]{hinman}, \cite[Secion III, Chapter 4]{kk}, \cite[Corollary 3.1.21]{marker}, \cite[Theorem 13.10]{monk} and  \cite[Section 4, Chapter III]{smorynski}. Here, we present a slightly different proof.
\begin{theorem}[Axiomatizablity of $\langle\mathbb{Z};<,+\rangle$]\label{thm-oaz}
The  infinite theory of non-trivial discretely ordered abelian groups with the division algorithm, that is $\texttt{O}_1$, $\texttt{O}_2$, $\texttt{O}_3$, $\texttt{A}_1$, $\texttt{A}_2$, $\texttt{A}_3$, $\texttt{A}_4$, $\texttt{A}_5$ and

\begin{tabular}{l}
($\texttt{O}_7^\circ$) \; $\forall x,y \big(x<y\leftrightarrow x+{\bf 1}\leqslant y\big)$ \\
($\texttt{A}_7^\circ$) \; $\forall x\exists  y\,\big(\bigvee\hspace{-1.5ex}\bigvee_{i<n}x=n\centerdot y + \bar{i}\,\big)$  \qquad \, \quad \, \qquad  $n\in\mathbb{N}^+$  \\
where $\bar{i}={\bf 1}+\cdots+{\bf 1}$ ($i-$\text{times})
\end{tabular}

\noindent
completely axiomatizes the order and additive theory of the   integer  numbers and, moreover, the \textup{(}theory of the\textup{)} structure $\langle\mathbb{Z};<,+,-,{\bf 0},{\bf 1},\{\equiv_n\}_{n\geqslant 2}\rangle$  admits quantifier elimination,   so has a decidable theory.
\end{theorem}
\begin{proof}
The axiom $\texttt{A}_7^\circ$ is  equivalent with 
 $\forall x\bigvee\hspace{-1.5ex}\bigvee_{i<n} \big(x\equiv_n\bar{i}\wedge \bigwedge\hspace{-1.5ex}\bigwedge_{i\neq j<n}x\not\equiv_n\bar{j}\big),$
and so the negation signs behind the congruences can be eliminated by the following equivalence:    $(a\not\equiv_n b) \leftrightarrow \bigvee\hspace{-1.5ex}\bigvee_{0<i<n} (a\equiv_n b+\bar{i}\,)$.
Whence, by Remark~\ref{mainlem}, it suffices to show the equivalence of the    formula
\begin{equation}\label{r-3}
\exists x (\bigwedge\hspace{-2.35ex}\bigwedge_{i<m} a_i\centerdot x\equiv_{n_i}t_i \;\wedge\;   \bigwedge\hspace{-2.35ex}\bigwedge_{j<p} u_j\!<\!b_j\centerdot x \;\wedge\;  \bigwedge\hspace{-2.3ex}\bigwedge_{k<q} c_k\centerdot x\!<\!v_k \;\wedge\;
\bigwedge\hspace{-2.25ex}\bigwedge_{\ell<r} d_\ell\centerdot x=w_\ell)
\end{equation}
with some quantifier-free formula, where $a_i$'s, $b_j$'s, $c_k$'s and $d_\ell$'s are natural numbers and $t_i$'s, $u_j$'s,  $v_k$'s and $w_\ell$'s are $x$-free terms. By the equivalences
\begin{itemize}\itemindent=5ex
\item[(i)] $[a<b] \leftrightarrow [n\centerdot a< n\centerdot b]$,
\item[(ii)] $[a=b] \leftrightarrow [n\centerdot a=n\centerdot b]$ and
\item[(iii)]  $[a\equiv_m b] \leftrightarrow [n\centerdot a\equiv_{nm}n\centerdot b]$
\end{itemize}
which are provable from the axioms, we can assume that $a_i$'s, $b_j$'s, $c_k$'s and $d_\ell$'s  in
the formula~\eqref{r-3} are equal to each other, say to $\alpha$. Now, \eqref{r-3} is equivalent with
\begin{equation}\label{r-4}
\exists y (y\equiv_\alpha{\bf 0} \;\wedge\; \bigwedge\hspace{-2.35ex}\bigwedge_{i<m} y\equiv_{n_i}t_i' \;\wedge\;   \bigwedge\hspace{-2.35ex}\bigwedge_{j<p} u_j'\!<\!y \;\wedge\;  \bigwedge\hspace{-2.3ex}\bigwedge_{k<q} y\!<\!v_k' \;\wedge\;
  \bigwedge\hspace{-2.25ex}\bigwedge_{\ell<r} y=w_\ell')
\end{equation}
for $y=\alpha\centerdot x$ and some (possibly new) terms $t_i'$'s, $u_j'$'s, $v_k'$'s and $w_\ell'$'s. If $r\neq 0$ then \eqref{r-4} is readily equivalent with the  quantifier-free formula which results from  substituting $w_0'$ with $y$.
So, it suffices to eliminate the quantifier of
\begin{equation}\label{r-5}
\exists x (\bigwedge\hspace{-2.35ex}\bigwedge_{i<m} x\equiv_{n_i}t_i \;\wedge\;   \bigwedge\hspace{-2.35ex}\bigwedge_{j<p} u_j\!<\!x \;\wedge\;  \bigwedge\hspace{-2.3ex}\bigwedge_{k<q} x\!<\!v_k).
\end{equation}
By the equivalence of the formula $\exists x(\theta(x)\wedge u_0\!<\!x \wedge u_1\!<\!x)$ with the following  formula
$\big[\exists x (\theta(x)\!\wedge\!u_0\!<\!x)\!\wedge\!u_1\!\leqslant\!u_0\big]
\!\vee\!\big[\exists x (\theta(x)\!\wedge\!u_1\!<\!x)\!\wedge\!u_0\!\leqslant\!u_1\big]$
we can assume that $p\leqslant 1$ (and   $q\leqslant 1$ by a dual argument).
Also,   $\exists x(\theta(x)\wedge x\equiv_{n_0}t_0  \wedge x\equiv_{n_1}t_1)$ is equivalent with
$\exists x(\theta(x)\wedge x\equiv_{n}t)  \wedge t_0\equiv_{d}t_1$ where $d$ is the greatest common divisor of $n_1$ and $n_2$, $n$ is their least common multiplier, and  $t=a_0(n_0/d)t_1 + a_1(n_1/d)t_0$ where $a_0,a_1$ satisfy B\'{e}zout's Identity $a_0n_0 + a_1n_1 = d$ (see the proof of Proposition~\ref{crt}). So, we can assume that $m\leqslant 1$ as well. Now, if $m=0$ then the formula~\eqref{r-5} is equivalent with a quantifier-free formula by Theorem~\ref{thm-oz} (with $\mathfrak{s}(x)=x+{\bf 1}$ just like the the way formula~\eqref{f-4} was equivalent with some quantifier-free formula). So, suppose that $m=1$. In this case, if any of $p$ or $q$ is equal to $0$ then \eqref{r-5} is equivalent with $\top$ (since any congruence can have infinitely large or infinitely small solutions). Finally, if $p=q=1=m$ then the formula
$
\exists x (x\equiv_{n}t \;\wedge\;   u\!<\!x \;\wedge\;  x\!<\!v)
$
is equivalent with $\exists y (r<n\centerdot y\leqslant s)$ for $x=t+n\centerdot y$, $r=u-t$ and $s=v-t-{\bf 1}$. Now,
$\exists y (r<n\centerdot y\leqslant s)$  is   equivalent with   the quantifier-free formula $\bigvee\hspace{-2.15ex}\bigvee_{i<n}(s\equiv_n \bar{i} \;\wedge\;  r+\bar{i}<s)$ since by the division algorithm there are some $q$ and some $i<n$ such that $s=qn + i$. The existence of some $y$ such that $r<ny\leqslant s$   is then equivalent with   $r<nq$ ($=s-i$).
\qed\end{proof}
\begin{remark}[\textbf{Infinite Axiomatizablity}]\label{rem-infz}
The theory of the structure $\langle\mathbb{Z};<,+\rangle$ cannot be axiomatized finitely, because $\texttt{O}_1$, $\texttt{O}_2$, $\texttt{O}_3$, $\texttt{A}_1$, $\texttt{A}_2$, $\texttt{A}_3$, $\texttt{A}_4$, $\texttt{A}_5$, $\texttt{O}_7^\circ$ and any finite number of the instances of $\texttt{A}_7^\circ$ cannot prove all the instances of $\texttt{A}_7^\circ$. To see this take $\mathfrak{p}$ to be a sufficiently large prime number and put $N=(\mathfrak{p}-1)!$. Let us recall that the set   $\mathbb{Q}/N=\{m/N^k\mid m\in\mathbb{Z},k\in\mathbb{N}\}$ of rational numbers is closed under addition and the operation $x\mapsto x/n$ for any $1<n<\mathfrak{p}$. Let $\mathcal{A}=(\mathbb{Q}/N)\times\mathbb{Z}$ and define  the structure $\mathfrak{A}=\langle\mathcal{A};<_\mathfrak{A},+_\mathfrak{A},
-_\mathfrak{A},{\bf 0}_\mathfrak{A},{\bf 1}_\mathfrak{A}\rangle$ by  the following:
\begin{itemize}\itemindent=2em
\item[$(<_\mathfrak{A})$:] $(a,\ell)<_\mathfrak{A}(b,m)\iff (a<b)\vee (a=b\wedge \ell<m)$;
\item[$(+_\mathfrak{A})$:] $(a,\ell)+_\mathfrak{A}(b,m)=(a+b,\ell+m)$;
\item[$(-_\mathfrak{A})$:] $-_\mathfrak{A}(a,\ell)=(-a,-\ell)$;
\item[$({\bf 0}_\mathfrak{A})$:] ${\bf 0}_\mathfrak{A}=(0,0)$;
\item[$({\bf 1}_\mathfrak{A})$:] ${\bf 1}_\mathfrak{A}=(0,1)$.
\end{itemize}
It is straightforward to see that $\mathfrak{A}$ satisfies the axioms $\texttt{O}_1$, $\texttt{O}_2$, $\texttt{O}_3$, $\texttt{A}_1$, $\texttt{A}_2$, $\texttt{A}_3$, $\texttt{A}_4$, $\texttt{A}_5$ and $\texttt{O}_7^\circ$; but does not satisfy $\texttt{A}_7^\circ$ for $n=\mathfrak{p}$ since we have    $(1,0)=\mathfrak{p}\centerdot(a,\ell)+\bar{i}$ for any element   $a\in\mathbb{Q}/N,\ell\in\mathbb{Z},i\in\mathbb{N}$ (with $i<\mathfrak{p}$) implies that $a=1/\mathfrak{p}$ but $1/\mathfrak{p}\not\in\mathbb{Q}/N$. However, $\mathfrak{A}$ satisfies the finite number of the instances of $\texttt{A}_7^\circ$ (for any $1<n<\mathfrak{p}$): for any $(a,\ell)\in\mathcal{A}$ we have $a=m/N^k$ for some $m\in\mathbb{Z}$, $k\in\mathbb{N}$, and $\ell=nq+r$ for  some $q,r$ with $0\leqslant r<n$; now, $(a,\ell)=n\centerdot \big(m'/N^{k+1},q\big)+_\mathfrak{A}(0,r)$ (where $m'=m\cdot (N/n)\in\mathbb{Z}$) and so $(a,\ell)=n\centerdot \big(m'/N^{k+1},q\big)+_\mathfrak{A}\bar{r}$ (where $\bar{r}={\bf 1}_\mathfrak{A}+_\mathfrak{A}\cdots+_\mathfrak{A}{\bf 1}_\mathfrak{A}$ for $r$ times).
\qedef\end{remark}
\begin{remark}[$\mathbf{\langle\mathbb{N};<,+\rangle}$]\label{rem-noa}
Since $\mathbb{N}$ is definable in the structure $\langle\mathbb{Z};<,+\rangle$  by 
$``x\in\mathbb{N}"\iff\exists y (y\!+\!y\!=\!y\wedge y\leqslant x)$,  we do not study 
 $\langle\mathbb{N};<,+\rangle$ separately (see \cite[Theorem~32E]{enderton}). In fact the decidability of $\langle\mathbb{Z};<,+\rangle$ implies the decidability of $\langle\mathbb{N};<,+\rangle$: relativization $\psi^{\mathbb{N}}$ of a $\{<,+\}$-formula $\psi$ resulted  from substituting any subformula of the form $\forall x\theta(x)$ by $\forall x [``x\in\mathbb{N}"\!\rightarrow\!\theta(x)]$ and $\exists x\theta(x)$ by $\exists x [``x\in\mathbb{N}"\!\wedge\!\theta(x)]$   has the following property: $\langle\mathbb{N};<,+\rangle\models\psi \iff \langle\mathbb{Z};<,+\rangle\models\psi^{\mathbb{N}}$.
\qedef\end {remark}

\section{The Multiplicative Ordered  Structures of Numbers}\label{sec-om}
In this final section we consider the theories of the number sets $\mathbb{N},\mathbb{Z},\mathbb{R}$ and $\mathbb{Q}$ in the language $\{<,\times\}$.
\subsection{\bf Natural Numbers with Order and Multiplication}\label{subsec-n}
The theory of the structure $\langle\mathbb{N};<,\times\rangle$ is not decidable (and so  no computably enumerable set of sentences can axiomatize this structure). This is because:
\begin{itemize}
\item[$\bullet$] The addition operation is definable in $\langle\mathbb{N};<,\times\rangle$, since
    \begin{itemize}
    \item[$\circ$] the successor operation $\mathfrak{s}$ is definable from order:
        \newline{$y\!=\!\mathfrak{s}(x) \iff x\!<\!y \wedge \neg\exists z (x\!<\!z\!<\!y),$}
    \item[$\circ$] and the addition operation is definable from the successor and multiplication:
        \newline{\qquad $z\!=\!x\!+\!y \quad \iff\quad  \big[\neg\exists u (\mathfrak{s}(u)\!=\!z)\wedge x\!=\!y\!=\!z\big] \vee$}
        \newline{\qquad $\big[\exists u (\mathfrak{s}(u)\!=\!z) \wedge \mathfrak{s}(z\cdot x)\cdot \mathfrak{s}(z\cdot y)=
\mathfrak{s}(z\cdot z\cdot \mathfrak{s}(x\cdot y))\big].$}
    \end{itemize}
This identity was first introduced by Robinson~\cite{robinson}; also see e.g.   \cite[Chapter 24]{bbj} or \cite[Exercise 2 on page 281]{enderton}.
\item[$\bullet$] Thus the structure $\langle\mathbb{N};<,\times\rangle$ can interpret  the structure  $\langle\mathbb{N};+,\times\rangle$ whose theory is  undecidable (see e.g. \cite[Theorem 17.4]{bbj}, \cite[Corollary 35A]{enderton}, \cite[Theorem 4.1.7]{hinman}, \cite[Chapter 15]{monk} or \cite[Corollary 6.4 in Chapter III]{smorynski}).
\end{itemize}

\subsection{\bf Integer Numbers with Order and Multiplication}\label{subsec-z}
The undecidability of the theory of the structure $\langle\mathbb{N};+,\times\rangle$ also implies the undecidability of the theories of the structures
$\langle\mathbb{Z};+,\times\rangle$ and $\langle\mathbb{Z};<,\times\rangle$ as follows:
\begin{itemize}
\item[$\bullet$] By Lagrange's Four Square Theorem (see e.g. \cite[Theorem 16.6]{monk}) $\mathbb{N}$ is definable in $\langle\mathbb{Z};+,\times\rangle$, and so $\langle\mathbb{Z};+,\times\rangle$ has an undecidable theory (see e.g.  \cite[Theorem 16.7]{monk} or \cite[Corollary 8.29 in Chapter III]{smorynski}).
\item[$\bullet$]
The following numbers and  operations   are definable in the structure $\langle\mathbb{Z};<,\times\rangle$:
   \begin{itemize}
   \item The number zero: $u=0 \iff \forall x (x\cdot u = u)$.
   \item The number one: $u=1 \iff \forall x (x\cdot u =x)$.
   \item The number $\texttt{-1}$: $u=-1 \iff u\cdot u=1 \wedge u\neq 1$.
   \item  The additive inverse: $y=-x \iff y=(-1)\cdot x$.
   \item The successor: $y=\mathfrak{s}(x) \iff x<y \wedge \neg\exists z (x<z<y)$.
   \item The addition: $z=x+y \iff [z=0\wedge y=-x] \quad \vee$

\hfill
$[z\neq 0\wedge \mathfrak{s}(z\cdot x)\cdot \mathfrak{s}(z\cdot y)=
\mathfrak{s}(z\cdot z\cdot \mathfrak{s}(x\cdot y))]$.
   \end{itemize}
There is another
beautiful definition for $+$ in terms of $\mathfrak{s}$ and $\times$ in $\mathbb{Z}$ at   \cite[p.~187]{hinman}:

  $z=x+y \iff$

\hfill
$[z\cdot\mathfrak{s}(z)=z\wedge\mathfrak{s}(x\cdot y)=\mathfrak{s}(x)\cdot\mathfrak{s}(y)]\vee[z\cdot\mathfrak{s}(z)\neq z\wedge \mathfrak{s}(z\cdot x)\cdot \mathfrak{s}(z\cdot y)=
\mathfrak{s}(z\cdot z\cdot \mathfrak{s}(x\cdot y))]$.
\item[$\bullet$] Whence, the structure $\langle\mathbb{Z};<,\times\rangle$ can interpret the undecidable  structure $\langle\mathbb{Z};+,\times\rangle$.
\end{itemize}

\subsection{\bf Real Numbers with Order and Multiplication}\label{subsec-r}
The structure $\langle\mathbb{R};<,\times\rangle$ is decidable, since by a theorem of Tarski the (theory of the)
structure $\langle\mathbb{R};<,+,\times\rangle$ can be completely axiomatized by the theory of {\em real closed ordered fields}, and so has a decidable theory; see e.g. \cite[Theorem 7, Chapter 4]{kk}, \cite[Theorem 3.3.15]{marker} or \cite[Theorem 21.36]{monk}. Here, we prove the decidability of the theory of $\langle\mathbb{R};<,\times\rangle$  directly (without using Tarski's theorem) and provide an explicit axiomatization for it.  Before that let us make a little note about the theory $\langle\mathbb{R}^+;<,\times\rangle$ (of the positive real numbers) which is isomorphic to $\langle\mathbb{R};<,+\rangle$ by the mapping $x\mapsto\log(x)$. Thus, we have  the following immediate corollary of Theorem~\ref{thm-oa}:

\begin{proposition}[Axiomatizablity of $\langle\mathbb{R}^+;<,\times\rangle$]\label{thm-or+}
The following infinite theory \textup{(}of non-trivial ordered divisible abelian groups\textup{)}
completely axiomatizes the order and multiplicative theory of the positive   real  numbers and, moreover, 
$\langle\mathbb{R}^+;<,\times,\square^{-1},{\bf 1}\rangle$  admits quantifier elimination, and so its theory is decidable.

\begin{tabular}{l}
($\texttt{O}_1$) \; $\forall x,y(x<y\rightarrow y\not<x)$  \\ ($\texttt{O}_2$) \; $\forall x,y,z (x<y<z\rightarrow x<z)$ \\
($\texttt{O}_3$) \; $\forall x,y (x<y \vee x=y \vee y<x)$ \\
($\texttt{M}_1$) \; $\forall x,y,z\,(x\cdot(y\cdot z)=(x\cdot y)\cdot z)$ \\
($\texttt{M}_2$) \; $\forall x (x\cdot\mathbf{1}=x)$ \\
($\texttt{M}_3$) \; $\forall x (x\cdot  x^{-1}=\mathbf{1})$ \\
($\texttt{M}_4$) \; $\forall x,y (x\cdot y=y\cdot x)$ \\
($\texttt{M}_5$) \; $\forall x,y,z(x<y\rightarrow x\cdot z<y\cdot z)$ \\
($\texttt{M}_6$) \; $\exists y (y\neq {\bf 1})$ \\
($\texttt{M}_7$) \; $\forall x\exists y(x=y^n)$  \qquad \qquad \qquad \qquad  $n\geqslant 2$
\end{tabular}

 \end{proposition}
\begin{proof}
For the infinite axiomatizability it suffices to note that for a sufficiently large $N$ the set $\{2^{{m}\cdot(N!)^{-k}}\mid m\in\mathbb{Z},k\in\mathbb{N}\}$ of positive real numbers (cf. Remark~\ref{rem-infq}) satisfies all the axioms ($\texttt{O}_1$, $\texttt{O}_2$, $\texttt{O}_3$, $\texttt{M}_1$,
 $\texttt{M}_2$, $\texttt{M}_3$, $\texttt{M}_4$, $\texttt{M}_5$,
 $\texttt{M}_6$) and  finitely many instances of the axiom $\texttt{M}_7$ (for $n\leqslant N$) but not all the instances of $\texttt{M}_7$ (for example when $n=p$ is a prime larger than $N!$).
\qed\end{proof}

\begin{theorem}[Axiomatizablity of $\langle\mathbb{R};<,\times\rangle$]\label{thm-or}
The following infinite theory
completely axiomatizes the order and multiplicative theory of the   real  numbers and, moreover, the structure $\langle\mathbb{R};<,\times,\square^{-1},{\bf -1},{\bf 0},{\bf 1}\rangle$  admits quantifier elimination, and so its theory is decidable.

\begin{tabular}{l}
($\texttt{O}_1$) \; $\forall x,y(x<y\rightarrow y\not<x)$  \\ ($\texttt{O}_2$) \; $\forall x,y,z (x<y<z\rightarrow x<z)$ \\
($\texttt{O}_3$) \; $\forall x,y (x<y \vee x=y \vee y<x)$ \\
($\texttt{M}_1$) \; $\forall x,y,z\,(x\cdot(y\cdot z)=(x\cdot y)\cdot z)$ \\
($\texttt{M}_2^\circ$) \; $\forall x (x\cdot\mathbf{1}=x  \;\,  \wedge \;\, x\cdot{\bf 0}={\bf 0}={\bf 0}^{-1})$ \\
($\texttt{M}_3^\circ$) \; $\forall x (x\neq {\bf 0}\rightarrow x\cdot  x^{-1}=\mathbf{1})$ \\
($\texttt{M}_4$) \; $\forall x,y (x\cdot y=y\cdot x)$ \\
($\texttt{M}_5^\circ$) \; $\forall x,y,z(x<y\wedge {\bf 0}<z \rightarrow x\cdot z<y\cdot z)$ \\
($\texttt{M}_5^\bullet$) \; $\forall x,y,z(x<y\wedge z<{\bf 0} \rightarrow y\cdot z<x\cdot z)$  \\
($\texttt{M}_6^\circ$) \; $\exists y ({\bf -1}<{\bf 0}<{\bf 1}<y)$
\\
($\texttt{M}_{7}^\circ$) \;  $\forall x\exists y (x=y^{2n+1})$
\\
($\texttt{M}_{8}$) \;  $\forall
x (x^{2n}={\bf 1}\longleftrightarrow x={\bf 1}\vee x={\bf -1})$ \\
($\texttt{M}_{9}$) \; $\forall x\,({\bf 0}<x\longleftrightarrow \exists y [y\neq{\bf 0}\wedge x=y^2])$
\end{tabular}

\end{theorem}
\begin{proof}
We have $(x<{\bf 0}) \leftrightarrow ({\bf 0}<-x)$ by $\texttt{M}_{5}^\bullet$, $\texttt{M}_{2}^\circ$, $\texttt{M}_{6}^\circ$ and $\texttt{M}_{8}$, where $-x=({\bf -1})\cdot x$.
Whence, for any quantifier-free formula $\eta$ we have  $$\exists x \eta(x)\equiv \exists x\!>\!{\bf 0}\eta(x)\vee \eta({\bf 0}) \vee \exists y\!>\!{\bf 0}\eta(-y).$$ Also, if $z$ is another variable in $\eta$ then $\eta(x,z)$ is equivalent with $[{\bf 0}<z \wedge \eta(x,z)]\vee \eta(x,{\bf 0}) \vee [{\bf 0}<-z\wedge\eta(x,z)]$. For the last disjunct, if we let $z'=-z$ then ${\bf 0}<-z\wedge\eta(x,z)$ will be ${\bf 0}<z'\wedge\eta(x,-z')$.
Thus, by introducing the constants ${\bf 0}$ and ${\bf -1}$ (and renaming the variables if necessary), we can assume that all the variables of a quantifier-free formula  are positive.
 Now, the process of eliminating the quantifier of the formula $\exists x\eta(x)$, where $\eta$ is the conjunction of some atomic formulas (cf. Remark~\ref{mainlem}) goes as follows:  we first eliminate the constants ${\bf 0}$ and ${\bf -1}$ and then reduce the desired conclusion to Proposition~\ref{thm-or+}.
 For the first part, we simplify terms so that each term is either positive (all the variables are positive) or equals to ${\bf 0}$ or is the negation of a positive term (is $-t$ for some positive term $t$). Then by replacing ${\bf 0}={\bf 0}$ with $\top$ and ${\bf 0}<{\bf 0}$ with $\bot$ we can assume that ${\bf 0}$ appears at most once in any atomic formula; also ${\bf -1}$ appears at most once since $-t=-s$ is equivalent with $t=s$ and $-t<-s$ with $s<t$.
 Now, we can eliminate the constant ${\bf -1}$ by replacing the atomic formulas $-t=s$, $t=-s$ and $t<-s$ by $\bot$ and $-t<s$ by $\top$ for positive or zero terms $t,s$ (note that   ${\bf -0}={\bf 0}$ by $\texttt{M}_{2}^\circ$).  Also the constant ${\bf 0}$ can be eliminated by replacing ${\bf 0}<t$ with $\top$ and  $t<{\bf 0}$ and $t={\bf 0}$ (also ${\bf 0}=t$) with $\bot$ for positive terms $t$.
 Thus, we get a formula whose all variables are positive, and so we are in the realm of $\mathbb{R}^+$. Finally, for the second part we have the equivalence of thus resulted formula with a quantifier-free formula by Proposition~\ref{thm-or+} provided that the relativized form of the axioms $\texttt{O}_{1}$, $\texttt{O}_{2}$, $\texttt{O}_{3}$, $\texttt{M}_{1}$, $\texttt{M}_{2}$, $\texttt{M}_{3}$, $\texttt{M}_{4}$, $\texttt{M}_{5}$, $\texttt{M}_{6}$ and $\texttt{M}_{7}$ to $\mathbb{R}^+$ can be proved from the axioms $\texttt{O}_{1}$, $\texttt{O}_{2}$, $\texttt{O}_{3}$, $\texttt{M}_{1}$, $\texttt{M}_{2}^\circ$, $\texttt{M}_{3}^\circ$, $\texttt{M}_{4}$, $\texttt{M}_{5}^\circ$, $\texttt{M}_{5}^\bullet$, $\texttt{M}_{6}^\circ$, $\texttt{M}_{7}^\circ$, $\texttt{M}_{8}$,  and $\texttt{M}_{9}$. We need to consider   $\texttt{M}_{6}$ and $\texttt{M}_{7}$ only, when relativized to $\mathbb{R}^+$, i.e., $\exists y({\bf 0}<y \wedge y\neq {\bf 1})$ and $\forall x \exists y [{\bf 0}<x \rightarrow {\bf 0}<y \wedge x=y^n]$. The relativization of $\texttt{M}_{6}$ immediately follows from  $\texttt{M}_{6}^\circ$. For the relativization of $\texttt{M}_{7}$ take any $a>{\bf 0}$, and any $n\in\mathbb{N}$. Write $n=2^k(2m+1)$; by $\texttt{M}_{7}^\circ$ there exists some $c$ such that $c^{2m+1}=a$, and by $\texttt{M}_{5}^\circ$ and $\texttt{M}_{5}^\bullet$ we should have $c>{\bf 0}$. Now, by using $\texttt{M}_{9}$  for $k$ times there must exist some $b$ such that $b^{2^k}=c$ and we can have $b>{\bf 0}$ (since otherwise we can take $-b$ instead of $b$). Now, we have $b^{2^k(2m+1)}=c^{2m+1}=a$ and so $a=b^n$.
 \qed\end{proof}

That no finite set of axioms can completely axiomatize the theory of $\langle\mathbb{R};<,\times\rangle$ can be seen from the fact that the set $\{0\}\cup\{-2^{{m}\cdot(N!)^{-k}},2^{{m}\cdot(N!)^{-k}}\mid m\in\mathbb{Z},k\in\mathbb{N}\}$ of real numbers, for some $N>2$, satisfies all the axioms of Theorem~\ref{thm-or} except $\texttt{M}_{7}^\circ$; however it satisfies a finite number of its instances (when $2n+1\leqslant N$) but not all the instances (e.g. when $2n+1$ is a prime greater than $N!$) of $\texttt{M}_{7}^\circ$ (cf. the proof of Proposition~\ref{thm-or+} and Remark~\ref{rem-infq}).

\subsection{\bf Rational Numbers with Order and Multiplication}\label{subsec-q}
The technique of the proof of Theorem~\ref{thm-or} enables us to consider first the multiplicative and order structure of the positive rational numbers $\langle\mathbb{Q}^+;<,\times\rangle$. One can easily see that the formula $\exists x(y=x^n)$ (for $n>1$) is not equivalent with any quantifier-free formula in $\langle\mathbb{Q}^+;<,\times\rangle$; so let us introduce the following notation.
\begin{definition}[$\Re$]\label{def-re}
Let  $\Re_n(y)$ be the formula $\exists x(y=x^n)$, stating that ``$y$ is the $n$th power of a number'' (for $n>1$).
\qedef\end{definition}
Now we can introduce our candidate axiomatization for the theory of the structure $\langle\mathbb{Q}^+;<,\times\rangle$.
\begin{definition}[${\sf TQ}$]\label{def-tq}
Let ${\sf TQ}$ be the theory axiomatized by the axioms $\texttt{O}_{1}$, $\texttt{O}_{2}$, $\texttt{O}_{3}$, $\texttt{M}_{1}$, $\texttt{M}_{2}$, $\texttt{M}_{3}$, $\texttt{M}_{4}$, $\texttt{M}_{5}$  and $\texttt{M}_{6}$ of Proposition~\ref{thm-or+}
plus the following two axiom schemes:

\begin{tabular}{l}
  $(\texttt{M}_{10})$ \;  $\forall x,z\exists y (x<z\rightarrow x<y^n<z)$, and  \\
 $(\texttt{M}_{11})$ \;  $\forall \{x_j\}_{j<q} \exists y \forall z  \bigwedge\hspace{-1.5ex}\bigwedge_{m_j\nmid n (j<q)} (y^n\cdot x_j \neq z^{m_j})$;
 \end{tabular}

\quad  for each $n\geqslant 1$ (and $m_j>1$).
\qedef
\end{definition}
Some explanations on the new axioms  $\texttt{M}_{10}$ and $\texttt{M}_{11}$ are in order. The axiom $\texttt{M}_{10}$, interpreted in $\mathbb{Q}^+$,  states that   $\mathbb{Q}^+$ is dense not only in itself but also in the radicals of its elements (or more generally in $\mathbb{R}^+$: for any $x,z\in\mathbb{Q}^+$ there exists some $y\in\mathbb{Q}^+$ that satisfies $\sqrt[n]{x}<y<\sqrt[n]{z}$). The axiom $\texttt{M}_{11}$, interpreted in $\mathbb{Q}^+$ again, is actually equivalent with the fact that for any sequences $x_1,\cdots,x_q\in\mathbb{Q}^+$ and  $m_1,\cdots,m_q\in\mathbb{N}^+$ none of which divides $n$ (in symbols $m_j\nmid n$), there exists some $y\in\mathbb{Q}^+$ such that $\bigwedge\hspace{-1.5ex}\bigwedge_{j}\neg\Re_{m_j}(y^n\cdot x_j)$. This axiom is not true in $\mathbb{R}^+$ (while $\texttt{M}_{10}$ is true in it) and to see that why $\texttt{M}_{11}$ is true in $\mathbb{Q}^+$ it suffices to note that for given $x_1,\cdots,x_q$ one can take $y$ to be a prime number which does not appear in the unique factorization (of the numerators and denominators of the reduced forms) of any of $x_j$'s. In this case $y^n\cdot x_j$ can be an $m_j$'s power (of a rational number) only when $m_j$ divides $n$. The condition $m_j\nmid n$ is necessary, since otherwise (if $m_j\mid n$ and) if $x_j$ happens to satisfy $\Re_{m_j}(x_j)$ then no   $y$ can satisfy the relation  $\neg\Re_{m_j}(y^n\cdot x_j)$.

We now show that ${\sf TQ}$ completely axiomatizes the theory of  the structure
$$\langle\mathbb{Q}^+;<,\times,\square^{-1},{\bf 1},\{\Re_n\}_{n>1}\rangle$$ and moreover this structure  admits quantifier elimination, thus the theory of the structure $\langle\mathbb{Q}^+;<,\times\rangle$ is decidable.
For that, we will need the following lemmas.
\begin{lemma}\label{lem-1}
For any $x\in\mathbb{Q}^+$ and any natural   $n_1,n_2>1$,
$$\Re_{n_1}(x)\wedge\Re_{n_2}(x)\iff \Re_n(x)$$
 where $n$ is the least common multiplier of $n_1$ and $n_2$.
\end{lemma}
\begin{proof}
Suppose   $x=y^{n_1}=z^{n_2}$ holds. By B\'{e}zout's Identity there are   $c_1,c_2\in\mathbb{Z}$ such that $c_1n/n_1+c_2n/n_2=1$; therefore,   $x=x^{c_1n/n_1}\cdot x^{c_2n/n_2}=y^{c_1n}\cdot z^{c_2n}=(y^{c_1}z^{c_2})^n$.
\qed\end{proof}

\begin{lemma}
\label{lem-1qe}
For natural numbers $\{n_i\}_{i<p}$ with $n_i>1$ and positive rational numbers $\{t_i\}_{i<p}$ and $x$,
$$\bigwedge\hspace{-2.25ex}\bigwedge_{i<p}\Re_{n_i}(x\cdot t_i) \iff \Re_n(x\cdot\beta)\wedge
\bigwedge\hspace{-2.15ex}\bigwedge_{i\neq j}\Re_{d_{i,j}}(t_i\cdot t_j^{-1})$$
 where $n$ is the least common multiplier of $n_i$'s, $d_{i,j}$ is the greatest common divisor of $n_i$ and $n_j$ \textup{(}for each $i\neq j$\textup{)} and $\beta=\prod_{i<p}t_i^{c_i(n/n_i)}$ in which $c_i$'s satisfy $\sum_{i<p}c_i(n/n_i)=1$.
\end{lemma}
\begin{proof}
For $t_i$'s, $n_i$'s,  $c_i$'s,  $d_{i,j}$'s and $n$  as given above, we show that   $\Re_{n_k}(t_k\cdot \beta^{-1})$ holds for each  fixed $k<p$ when $\bigwedge\hspace{-1.5ex}\bigwedge_{i\neq j}\Re_{d_{i,j}}(t_i\cdot t_j^{-1})$ holds.   Let $m_{k,i}$ be  the least common multiplier of $n_k$ and $n_i$ (which is a divisor of $n$ then). Let us note that $d_{k,i}/n_i=n_k/m_{k,i}$.
Since $\Re_{d_{k,i}}(t_k\cdot t_i^{-1})$ there should exists some $w_{k,i}$'s (for $i\neq k$) such that $t_k\cdot t_i^{-1}=w_{k,i}^{d_{k,i}}$.
Now, the relation  $\Re_{n_k}(t_k\cdot \beta^{-1})$ follows from the following identities:
 $t_k\cdot \beta^{-1}=t_k^{\sum_{i}c_i(n/n_i)}\cdot \prod_{i}t_i^{-c_i(n/n_i)}=\prod_{i\neq k}(t_k\cdot t_i^{-1})^{c_i(n/n_i)}=\prod_{i\neq k} (w_{k,i}^{d_{k,i}})^{c_i(n/n_i)} = \prod_{i\neq k} w_{k,i}^{c_i\cdot n_k(n/m_{k,i})} = (\prod_{i\neq k} w_{k,i}^{c_i(n/m_{k,i})})^{n_k}$.
\begin{itemize}\itemindent=2.5ex
\item[($\Rightarrow$):]   The relations $\Re_{n_i}(x\cdot t_i)$ and   $\Re_{n_j}(x\cdot t_j)$ immediately imply that  $\Re_{d_{i,j}}(x\cdot t_i)$ and $\Re_{d_{i,j}}(x\cdot t_j)$ and so $\Re_{d_{i,j}}(t_i\cdot t_j^{-1})$. For showing $\Re_n(x\cdot\beta)$ it suffices, by Lemma~\ref{lem-1}, to show that $\Re_{n_i}(x\cdot \beta)$ holds for each $i<p$.   This immediately follows from   the relation  $\Re_{n_i}(t_i\cdot \beta^{-1})$ which was proved above, and the assumption $\Re_{n_i}(x\cdot t_i)$.
\item[($\Leftarrow$):]
From  the first part of the proof we have   $\Re_{n_k}(t_k\cdot \beta^{-1})$ for each $k<p$; now by $\Re_n(x\cdot \beta)$ we have $\Re_{n_k}(x\cdot \beta)$ and so $\Re_{n_k}(x\cdot t_k)$  for each $k<p$.
\qed
\end{itemize}
\end{proof}
Let us note that Lemmas~\ref{lem-1} and~\ref{lem-1qe} are provable in   ${\sf TQ}$. The idea of the proof of Lemma~\ref{lem-1qe} is taken from \cite{ore}.
\begin{lemma}\label{lem-aqe}
The following sentences are provable in ${\sf TQ}$ for any $n>1$\textup{:}

$\forall u\exists y [\Re_n(y\cdot u)]$,

$\forall x,u\exists y [x<y\wedge\Re_n(y\cdot u)]$,

$\forall z,u\exists y [y<z\wedge\Re_n(y\cdot u)]$ and

$\forall x,z,u\exists y [x<z\rightarrow x<y<z\wedge\Re_n(y\cdot u)]$.
\end{lemma}
\begin{proof}
We show the last formula only. By $\texttt{M}_{10}$ (of Definition~\ref{def-tq}) there exists   $v$ such that $x\cdot u<v^n<z\cdot u$. Then for $y=v^n\cdot u^{-1}$ we will have $x<y<z$ and $\Re_n(y\cdot u)$.
\qed\end{proof}
\begin{lemma}\label{lem-bqe}
The following sentences are provable in ${\sf TQ}$ for any $m_1,\cdots,m_j,\cdots>1$\textup{:}

$\forall \{x_j\}_{j<q}\exists y [\bigwedge\hspace{-1.5ex}\bigwedge_{j<q}\neg\Re_{m_j}(y\cdot x_j)]$,

$\forall \{x_j\}_{j<q},u\exists y [u<y\wedge\bigwedge\hspace{-1.5ex}\bigwedge_{j<q}\neg\Re_{m_j}(y\cdot x_j)]$,

$\forall \{x_j\}_{j<q},v\exists y [y<v\wedge \bigwedge\hspace{-1.5ex}\bigwedge_{j<q}\neg\Re_{m_j}(y\cdot x_j)]$ and

$\forall \{x_j\}_{j<q},u,v\exists y [u<v\rightarrow u<y<v\wedge \bigwedge\hspace{-1.5ex}\bigwedge_{j<q}\neg\Re_{m_j}(y\cdot x_j)]$.
\end{lemma}
\begin{proof}
The first sentence is a  consequence of $\texttt{M}_{11}$ (of Definition~\ref{def-tq}) for $n=1$. We show the last sentence. There exists  $\gamma$, by $\texttt{M}_{11}$, such that $\bigwedge\hspace{-1.5ex}\bigwedge_{j}\neg\Re_{m_j}(\gamma\cdot x_j)$. Let $M=\prod_jm_j$; by $\texttt{M}_{10}$ there exists $\delta$ such that $u\cdot\gamma^{-1}<\delta^M<v\cdot\gamma^{-1}$. Now for $y=\gamma\cdot\delta^M$ we have that $u<y<v$ and $\bigwedge\hspace{-1.5ex}\bigwedge_{j}\neg\Re_{m_j}(y\cdot x_j)$ since if (otherwise) we had $\Re_{m_j}(y\cdot x_j)$ then $\Re_{m_j}(\gamma\cdot\delta^M\cdot x_j)$ and so $\Re_{m_j}(\gamma\cdot x_j)$ would hold; a contradiction.
\qed\end{proof}
\begin{lemma}\label{lem-cqe}
In the theory ${\sf TQ}$ the following formulas

$\exists x [\Re_n(x\cdot t)\wedge \bigwedge\hspace{-1.5ex}\bigwedge_{j<q}\neg\Re_{m_j}(x\cdot s_j)]$,

$\exists x [u<x\wedge \Re_n(x\cdot t)\wedge \bigwedge\hspace{-1.5ex}\bigwedge_{j<q}\neg\Re_{m_j}(x\cdot s_j)]$ and

$\exists x [x<v\wedge \Re_n(x\cdot t)\wedge \bigwedge\hspace{-1.5ex}\bigwedge_{j<q}\neg\Re_{m_j}(x\cdot s_j)]$

\noindent
are equivalent with

\qquad $\bigwedge\hspace{-1.5ex}\bigwedge_{m_j\mid n (j<q)}\neg\Re_{m_j}(t^{-1}\cdot s_j)$\textup{;}

\noindent
and the formula

$\exists x [u<x<v\wedge \Re_n(x\cdot t)\wedge \bigwedge\hspace{-1.5ex}\bigwedge_{j<q}\neg\Re_{m_j}(x\cdot s_j)]$

\noindent
is equivalent with

\qquad $\bigwedge\hspace{-1.5ex}\bigwedge_{m_j\mid n (j<q)}\neg\Re_{m_j}(t^{-1}\cdot s_j)\wedge u<v$.
\end{lemma}
\begin{proof}
If $m_j\mid n$ then $\Re_n(x\cdot t)$ implies $\Re_{m_j}(x\cdot t)$. Now, if $\Re_{m_j}(t^{-1}\cdot s_j)$ were true then $\Re_{m_j}(x\cdot s_j)$ would be true too;  contradicting $\bigwedge\hspace{-1.5ex}\bigwedge_{j<q}\neg\Re_{m_j}(x\cdot s_j)$. Suppose now that the relation $\bigwedge\hspace{-1.5ex}\bigwedge_{m_j\mid n}\neg\Re_{m_j}(t^{-1}\cdot s_j)$ holds. By $\texttt{M}_{11}$ there exists some $\gamma$ such that $\bigwedge\hspace{-1.5ex}\bigwedge_{m_j\nmid n}\neg\Re_{m_j}(\gamma\cdot t^{-1}\cdot s_j)$. By $\texttt{M}_{10}$ there exists some $\delta$ such that $u\cdot t\cdot\gamma^{-n}<\delta^{M\cdot n}<v\cdot t\cdot\gamma^{-n}$ (if $u<v$) where $M=\prod_{j<q}m_j$. For $x=\delta^{M\cdot n}\cdot \gamma^{n}\cdot t^{-1}$ we have $u<x<v$ and $\Re_n(x\cdot t)$. We show $\neg\Re_{m_j}(x\cdot s_j)$ for each $j<q$ by distinguishing two cases: if $m_j\mid n$ then $\neg\Re_{m_j}(t^{-1}\cdot s_j)$ implies $\neg\Re_{m_j}(\delta^{M\cdot n}\cdot \gamma^{n}\cdot t^{-1}\cdot s_j)$; if $m_j\nmid n$ then $\neg\Re_{m_j}(\gamma\cdot t^{-1}\cdot s_j)$ implies $\neg\Re_{m_j}(\delta^{M\cdot n}\cdot \gamma^{n}\cdot t^{-1}\cdot s_j)$.
\qed\end{proof}
\begin{theorem}[Axiomatizablity of $\langle\mathbb{Q};<,\times\rangle$]\label{omq+}
The infinite theory ${\sf TQ}$ completely axiomatizes the theory of   $\langle\mathbb{Q}^+;<,\times\rangle$, and moreover $\langle\mathbb{Q}^+;<,\times,\square^{-1},{\bf 1},\{\Re_n\}_{n>1}\rangle$   admits quantifier elimination.

\noindent
Also, the structure $\langle\mathbb{Q};<,\times\rangle$ can be completely axiomatized by the theory that results from ${\sf TQ}$ by adding the axioms $\texttt{M}_8$ \textup{(}in  Theorem~\ref{thm-or}\textup{)} and
substituting its $\texttt{M}_2$, $\texttt{M}_3$, $\texttt{M}_5$,  $\texttt{M}_6$  and  $\texttt{M}_{10}$, respectively, with the axioms $\texttt{M}_2^\circ$, $\texttt{M}_3^\circ$, $\texttt{M}_5^\circ$, $\texttt{M}_5^\bullet$, $\texttt{M}_6^\circ$  and

 $(\texttt{M}_{10}^\circ)$    $\forall x,z\exists y ({\bf 0}<x<z\rightarrow x<y^n<z)$.

\noindent
 Moreover, 
   $\langle\mathbb{Q};<,\times,\square^{-1},{\bf -1},{\bf 0},{\bf 1},\{\Re_n\}_{n>1}\rangle$   admits quantifier elimination.
\end{theorem}
\begin{proof}
Let us prove the $\mathbb{Q}^+$ part only. We are to eliminate the quantifier of the formula
\begin{equation}\label{q-1}
\exists x (\bigwedge\hspace{-2.35ex}\bigwedge_{i<p} \Re_{n_i}(x^{a_i}\cdot t_i) \;\wedge\; \bigwedge\hspace{-2.35ex}\bigwedge_{j<q} \neg\Re_{m_j}(x^{b_j}\cdot s_j) \;\wedge\;
 \bigwedge\hspace{-2.35ex}\bigwedge_{k<f} u_k\!<\!x^{c_k} \;\wedge\;  \bigwedge\hspace{-2.3ex}\bigwedge_{\ell<g} x^{d_\ell}\!<\!v_\ell \;\wedge\;
 \bigwedge\hspace{-2.25ex}\bigwedge_{\iota<h} x^{e_\iota}=w_\iota).
\end{equation}
By the equivalences $a^n<b^n \leftrightarrow a<b$ and $\Re_{m\cdot n}(a^n)\leftrightarrow \Re_m(a)$ we can assume that all the $a_i$'s, $b_j$'s, $c_k$'s, $d_\ell$'s and $e_\iota$'s are equal to each other, and moreover, equal to one (cf. the proof of Theorem~\ref{thm-oaz}). We can also assume that $h=0$ and that $f,g\leqslant 1$. By Lemma~\ref{lem-1qe}   we can also assume that $p\leqslant 1$. If $q=0$ then Lemma~\ref{lem-aqe} implies that the quantifier of the  formula~\eqref{q-1} can be eliminated. So, we assume that $q>0$. If $p=0$ then the quantifier of \eqref{q-1} can be eliminated by Lemma~\ref{lem-bqe}. Finally, if $p=1$ (and $q\neq 0=h$ and  $f,g\leqslant 1$) then Lemma~\ref{lem-cqe} implies that the formula  \eqref{q-1} is equivalent with a quantifier-free formula.
\qed\end{proof}
\begin{corollary}[Non-Definability of Addition]\label{cor-notsum}
The addition operation $(+)$ is not definable in the structure  $\langle\mathbb{Q};<,\times\rangle$.
\end{corollary}
\begin{proof}
If it were, then the structure $\langle\mathbb{Q};<,+,\times\rangle$ would be decidable by Theorem~\ref{omq+}; but Robinson~\cite{robinson} proved that this structure is not decidable.
\qed\end{proof}
\begin{remark}[\textbf{Infinite Axiomatizability}]\label{rem-q}
To see that the structure $\langle\mathbb{Q}^+;<,\times\rangle$ cannot be finitely axiomatized, we present an ordered multiplicative structure that satisfies any sufficiently large  finite number of the axioms of ${\sf TQ}$ but does not satisfy all of its axioms. Let $\mathfrak{p}$ be a sufficiently large prime number. Let us recall that the set $\mathbb{Q}/\mathfrak{p}=\{m/\mathfrak{p}^k\mid m\in\mathbb{Z},k\in\mathbb{N}\}$ is closed under addition and the operation $x\mapsto x/\mathfrak{p}$,  and 
$\mathbb{Z}\subset\mathbb{Q}/\mathfrak{p}\subset\mathbb{Q}$ holds.
Let $\rho_0,\rho_1,\rho_2,\cdots$ denote  the sequence of all prime numbers ($2,3,5,\cdots$). Let $(\mathbb{Q}/\mathfrak{p})^\ast$ be the set $\{\prod_{i<\ell}\rho_i^{r_i}\mid \ell\in\mathbb{N},r_i\in\mathbb{Q}/\mathfrak{p}\}$;  this  is closed under multiplication and the operation $x\mapsto {x}^{1/\mathfrak{p}}$, and we have 
$\mathbb{Q}^+\subset (\mathbb{Q}/\mathfrak{p})^\ast \subset\mathbb{R}^+$. Thus, $(\mathbb{Q}/\mathfrak{p})^\ast$  satisfies the axioms $\texttt{O}_{1}$, $\texttt{O}_{2}$, $\texttt{O}_{3}$, $\texttt{M}_{1}$, $\texttt{M}_{2}$, $\texttt{M}_{3}$, $\texttt{M}_{4}$, $\texttt{M}_{5}$  and $\texttt{M}_{6}$ of Proposition~\ref{thm-or+}, and also the axiom $\texttt{M}_{10}$. However, it does not satisfy the axiom $\texttt{M}_{11}$ for $n=q=x_0=1$ and $m_0=\mathfrak{p}$ because $(\mathbb{Q}/\mathfrak{p})^\ast\models\forall y\Re_{\mathfrak{p}}(y)$. We show that $(\mathbb{Q}/\mathfrak{p})^\ast$ satisfies the instances of the axiom $\texttt{M}_{11}$ when $1<m_j<\mathfrak{p}$ (for each $j<q$  and arbitrary $n,q$). Thus, no finite number of the instances of $\texttt{M}_{11}$ can prove all of its instances (with the rest of the axioms of ${\sf TQ}$). Let $x_j$'s be given from $(\mathbb{Q}/\mathfrak{p})^\ast$; write $x_j=\prod_{i<\ell_j}\rho_i^{r_{i,j}}$ where we can assume that $\ell_j\geqslant q$. Put $r_{j,j}=u_j/\mathfrak{p}^{v_j}$ where $u_j\in\mathbb{Z}$ and $v_j\in\mathbb{N}$ (for each $j<q$).
Define $t_j$ to be $1$ when $m_j\mid u_j$ and be $m_j$ when $m_j\nmid u_j$. Let $y=\prod_{i<q}\rho_i^{(t_i/\mathfrak{p}^{v_i+1})}$ ($\in(\mathbb{Q}/\mathfrak{p})^\ast$). We show $\bigwedge\hspace{-1.5ex}\bigwedge_{j<q}\neg\Re_{m_j}(y^n\cdot x_j)$ under the assumption $\bigwedge\hspace{-1.5ex}\bigwedge_{j<q} m_j\nmid n$. Take a $k<q$, and assume (for the sake of contradiction) that $\Re_{m_k}(y^n\cdot x_k)$. Then $\Re_{m_k}(\rho_k^{nt_k/\mathfrak{p}^{v_k+1}}\cdot
\rho_k^{u_k/\mathfrak{p}^{v_k}})$ holds, and so there should exist some $a,b$ 
such that $\rho_k^{(nt_k+\mathfrak{p}u_k)/\mathfrak{p}^{v_k+1}}=
\rho_k^{(m_k\cdot a)/\mathfrak{p}^b}$. Therefore, $m_k\mid nt_k+\mathfrak{p}u_k$. We reach to a contradiction by distinguishing two cases:

\hspace{-1em} (i) if $m_k\mid u_k$ then $t_k=1$ and so $m_k\mid n+\mathfrak{p}u_k$ whence $m_k\mid n$,  contradicting $\bigwedge\hspace{-1.5ex}\bigwedge_{j<q} m_j\nmid n$;

\hspace{-1em} (ii) if $m_k\nmid u_k$ then $t_k=m_k$ and so $m_k\mid nm_k+\mathfrak{p}u_k$ whence $m_k\mid \mathfrak{p}u_k$ which by the identity $(m_k,\mathfrak{p})=1$ implies that $m_k\mid u_k$, contradicting the assumption (of $m_k\nmid u_k$).
\qedef\end{remark}
\section{Conclusions}\label{sec-conc}
In the following table the decidable structures are denoted by $\boldsymbol\triangle$ and the undecidable ones by $\boldsymbol\triangle\hspace{-2.45mm}\backslash\hspace{-1.75mm}/$ :
\begin{center}
\begin{tabular}{|c||c|c|c|c|}
\hline
  & $\mathbb{N}$ & $\mathbb{Z}$ & $\mathbb{Q}$ & $\mathbb{R}$ \\
\hline
\hline
$\{<\}$ & $\boldsymbol\triangle$ & $\boldsymbol\triangle$ & $\boldsymbol\triangle$ & $\boldsymbol\triangle$ \\
\hline
$\{<,+\}$ & $\boldsymbol\triangle$ & $\boldsymbol\triangle$ & $\boldsymbol\triangle$ & $\boldsymbol\triangle$ \\
\hline
$\{<,\times\}$ & $\boldsymbol\triangle\hspace{-2.45mm}\backslash\hspace{-1.75mm}/$ \;  & $\boldsymbol\triangle\hspace{-2.45mm}\backslash\hspace{-1.75mm}/$ \;  &   $\boldsymbol\triangle$     & $\boldsymbol\triangle$ \\
\hline
\hline
$\{+,\times\}$ & $\boldsymbol\triangle\hspace{-2.45mm}\backslash\hspace{-1.75mm}/$ \;  & $\boldsymbol\triangle\hspace{-2.45mm}\backslash\hspace{-1.75mm}/$ \;  &  \ $\boldsymbol\triangle\hspace{-2.45mm}\backslash\hspace{-1.75mm}/$ \;   & $\boldsymbol\triangle$ \\
\hline
\end{tabular}
\end{center}
The decidability of the structure $\langle\mathbb{Q};<,\times\rangle$ is a new result of this paper, along with the explicit axiomatization for the already known decidable structure  $\langle\mathbb{R};<,\times\rangle$.
For the other decidable structures (other than $\langle\mathbb{N};<\rangle$ and $\langle\mathbb{N};<,+\rangle$) some old and some new (syntactic) proofs were given  for their decidability, with explicit axiomatizations. It is interesting to note that the undecidability of $\langle\mathbb{N};<,\times\rangle$ and $\langle\mathbb{Z};<,\times\rangle$ are inherited from the undecidability of $\langle\mathbb{N};+,\times\rangle$ and $\langle\mathbb{Z};+,\times\rangle$ (and the definability of $+$ in terms of $<$ and $\times$ in $\mathbb{N}$ and $\mathbb{Z}$), and the decidability of $\langle\mathbb{R};<,\times\rangle$ comes from the decidability of $\langle\mathbb{R};+,\times\rangle$ (and the definability of $<$ in terms of $+$ and $\times$ in $\mathbb{R}$). Nonetheless, the undecidability of the structure $\langle\mathbb{Q};+,\times\rangle$ has nothing to do with the (decidable) structure $\langle\mathbb{Q};<,\times\rangle$; indeed $+$ is not definable in $\langle\mathbb{Q};<,\times\rangle$ even though $<$ is definable in $\langle\mathbb{Q};+,\times\rangle$.


%
%
%
%
%
%
%


\end{document}